\documentclass[11pt]{article}

\usepackage[utf8]{inputenc}
\usepackage[T1]{fontenc}
\usepackage[margin=1in]{geometry}
\usepackage{amsmath, amsfonts, amssymb, amsthm}
\usepackage[none]{hyphenat}
\usepackage{graphicx}
\usepackage{float}
\usepackage{times}
\usepackage{lipsum}
\usepackage{authblk}
\usepackage{blindtext}

\newcommand{\R}{\mathbb R}

\newtheorem{Th}{Theorem}[section]
\numberwithin{Th}{section} 
\newtheorem{df}[Th]{Definition}
\newtheorem{pro}[Th]{Proposition}
\newtheorem{lem}[Th]{Lemma}

\newtheorem{cor}[Th]{Corollary}

\newtheorem{exs}[Th]{Examples}
\newtheorem{rma}[Th]{Remark}
\newtheorem{rms}[Th]{Remarks}

\DeclareMathOperator*{\esssup}{ess\,sup}
\DeclareMathOperator*{\essinf}{ess\,inf}

\newcommand{\norme}[1]{\left\Vert #1\right\Vert}

\def\Xint#1{\mathchoice
{\XXint\displaystyle\textstyle{#1}}%
{\XXint\textstyle\scriptstyle{#1}}%
{\XXint\scriptstyle\scriptscriptstyle{#1}}%
{\XXint\scriptscriptstyle\scriptscriptstyle{#1}}%
\!\int}
\def\XXint#1#2#3{{\setbox0=\hbox{$#1{#2#3}{\int}$ }
\vcenter{\hbox{$#2#3$ }}\kern-.6\wd0}}

\def\dashint{\Xint-}

\numberwithin{equation}{section}

\providecommand{\keywords}[1]{\textbf{\textit{Keywords.}} #1}
\providecommand{\subclass}[1]{\textbf{\textit{2010 Mathematics Subject Classification.}} #1}
\begin{document}

\title{\textbf{Wolff potential estimates for supersolutions of equations with generalized Orlicz growth}}
\author{Allami Benyaiche}
\affil{\small{Department of Mathematics, Ibn Tofail University, B.P: 133, Kenitra, Morocco\\
              Laboratory: PDE, Algebra and Spectral Geometry\\
              \it{allami.benyaiche@uit.ac.ma}}}
\author{Ismail Khlifi}
\affil{\small{Department of Mathematics, Ibn Tofail University, B.P: 133, Kenitra, Morocco\\
              Laboratory: PDE, Algebra and Spectral Geometry\\
              \it{is.khlifi@gmail.com}}}

\date{}

\maketitle

\textbf{Abstract.} In this paper, we establish pointwise estimates for supersolutions of quasilinear elliptic equations with structural conditions involving a generalized Orlicz growth in terms of a Wolff type potential. As a consequence, under the extra assumption, we obtain that the supersolutions satisfy a Harnack inequality and local Hölder continuity.\\

\keywords{Wolff potential $\cdot$ Generalized Orlicz growth $\cdot$ Generalized $\Phi$-function $\cdot$ Generalized Orlicz-Sobolev space $\cdot$ Supersolutions $\cdot$ Superharmonic functions}\\
\\
\subclass{31C45 $\cdot$ 31C15 $\cdot$ 35J62 $\cdot$ 35C15}

\section{Introduction}
Let $u$ be a (weak) supersolution of the quasilinear elliptic equation
\begin{equation}
- \text{div} \mathcal{A}(x,\nabla u)= 0,
\end{equation}
in a bounded open set $\Omega \subset \R^n$, where $\mathcal{A}$: $\Omega \times \R^n \rightarrow \R^n$ is a Carathéodory regular vector field satisfies the generalized Orlicz growth
$$ \mathcal{A}(x, \xi) \cdot \xi \approx  G(x,|\xi|),$$
for some generalized $\Phi$-functions $G(\cdot)$ (see section $2$). Then there is a Radon measure $\mu \geq 0$ such that the equation
 \begin{equation}
- \text{div} \mathcal{A}(x,\nabla u)= \mu,
\end{equation}
is satisfied by $u$ in the weak sense. In the linear classical potential theory, this supersolution is given using the Green potential $G^\mu_\Omega$ of $\mu$, precisely, we have
$$u(x) = G^\mu_\Omega(x) + h(x),$$
Where $h$ is a harmonic function. In contrast with this linear situation, this representation is not available in the nonlinear situation as in the $p$-Laplacian case:  $G(x,t) = t^p$. Indeed, by the fundamental works of Kilpeläinen and Malý in $\cite{ref14,ref15}$, we only have pointwise estimates for supersolutions in terms of Wolff potential
\begin{equation}
c_1W^\mu_{1,p}(x,R) \leq u(x) \leq c_2 \inf_{B(x,R)}u + c_3W^\mu_{1,p}(x,2R),
\end{equation}
where $B(x,R) \subset \Omega$, $\mu$ is the Radon measure associated to the supersolution $u$ and $W^\mu_{1,p}(x,R)$ is the Wolff potential defined by
\begin{equation}
W^\mu_{1,p}(x,R) = \int_0^R \left(\frac{\mu(B(x,s)}{s^{n-p}}\right)^\frac{1}{p-1} \, \frac{\mathrm{d}s}{s}.
\end{equation}

Next, Trudinger and Wang $\cite{ref26}$ gave a new method based on Poisson modification and Harnack inequality. Mikkonen has treated the weighted situation in $\cite{ref23}$. Björn and Björn $\cite{ref5}$, and Hara  $\cite{ref8}$ have developed the proof of potential estimates in the metric measure spaces. A new proof has been offered by Kuusi and Mingione $\cite{ref29}$, which allows the covering of general signed measures, not necessarily in the dual space, where their approach considers Solutions Obtained by Limits of Approximations (SOLA) instead of considering weak solutions.

In the variable exponent case, $G(x,t) = t^{p(x)}$, Alkhutov and Krasheninnikova in $\cite{ref1}$ and Lukkari,  Maeda and Marola in $\cite{ref20}$ gave a proof of the two-side Wolff potential which defined by.
\begin{equation}
 W^\mu_{p(\cdot)}(x,R) = \displaystyle \int_0^R \left(\frac{\mu(B(x,s)}{s^{n-p(x)}}\right)^\frac{1}{p(x)-1} \, \frac{\mathrm{d}s}{s} = \displaystyle \int_0^R \left(\frac{\mu(B(x,s)}{s^{n-1}}\right)^\frac{1}{p(x)-1} \, \mathrm{d}s.
\end{equation}
For the Orlicz case, $G(x,t) = G(t)$, the problem has been studied by Malý in $\cite{ref21}$ and recently by  Chlebicka,  Giannetti and Zatorska-Goldstein in $\cite{ref6}$, with
\begin{equation}
 W^\mu_{G}(x,R) = \displaystyle \int_0^R g^{-1}\left(\frac{\mu(B(x,s)}{s^{n-1}}\right) \, \mathrm{d}s,
\end{equation}
where $g$ is the right-hand derivative of $G$. For further informations about the Wolff potential estimates and its extensions, we refer to $\cite{ref12,ref16,ref18,ref22,ref24}$.\\

In this paper, we give the Wolff potential estimates of supersolutions of the equation $(1.1)$ under structural conditions involving a generalized $\Phi$-function $G(\cdot)$. For this, as in $(1.5)$ and $(1.6)$, the natural definition of the Wolff potential associated with $G(\cdot)$ and $\mu$ is given by
\begin{equation}
 W^{\mu}_{G(\cdot)}(x,R) := \int_0^R g^{-1}\left( x,\frac{\mu(B(x,s)}{s^{n-1}}\right) \, \mathrm{d}s,
\end{equation}
where $g(x,\cdot)$ is the right-hand derivative of $G(x,\cdot)$.\\
For our potential estimates, the major difficulty is that the function $G(\cdot,t)$ is just measurable, which does not allow us to choose test functions containing $ G(\cdot,t) $. The idea that makes the proof possible is using the upper envelope function $G^+$ and lower $G^-$ and the condition $(A_{1,n})$ (see section $2$) to link these two functions. Note that this condition plays the role of the logarithmic Hölder continuity in the variable exponent case. Using the Lorentz norm, Harnack estimates, and other techniques, we establish pointwise estimates for supersolutions of such equations in terms of the Wolff potential $(1.7)$, which also gives us another approach different from the previous methods.\\
Our main result is the following.
\begin{Th}
Let $G(\cdot) \in \Phi(\Omega)$ satisfies $(SC)$ and $(A_0)$. Let $u$ be a nonnegative supersolution to $(1.1)$ in $B_{2R} = B(x_0, 2R)$ and  $\mu$ is the associated Radon measure of $u$. If $u$ is lower semicontinuous at $x_0$ and that one of the following holds
\begin{enumerate}
\item[(1)] $G(\cdot)$ satisfies $(A_{1,s_*})$ and $\norme{u}_{L^s(B_{2R})} \leq d$, where $s_*= \frac{ns}{n+s}$ and  $s > \max\{\frac{n}{g_0},1\}(g^0-g_0)$.
\item[(2)] $G(\cdot)$ satisfies $(A_{1})$, $\norme{u}_{W^{1,G(\cdot)}(B_{2R})} \leq d$ and $\frac{ng_0}{n-g_0} > g^0$.
\end{enumerate} 
Then there exists a positive constant $C$ such that
\begin{equation}
\frac{1}{C} W^{\mu}_{G(\cdot)}(x_0,R) - 2R \leq u(x_0) \leq C\left(R + \inf_{B}u +  W^{\mu}_{G(\cdot)}(x_0,2R) \right).
\end{equation}
\end{Th}

Noting that, the conditions $(1)$ and $(2)$ are added to ensure the existence of weak Harnack inequality (see $\cite{ref2}$).\\

As an application of the Wolff potential estimates and under an assumption on the growth order of the measure $\mu$, we prove that the supersolutions satisfy a Harnack inequality and local Hölder continuity (see Theorem $5.14$ and Corollary $5.15$).\\

The paper is organized as follows. In Section $2$,  we give some properties of generalized $\Phi$-functions and Musielak-Orlicz-Sobolev spaces. In Section $3$,  we introduce weak solutions and the weak comparison principle. In Section $4$,  we use the monotone operator's theory to prove solutions to the Dirichlet problem with Sobolev Boundary values. In Section $5$, we establish lower and upper pointwise estimates for supersolutions in terms of the  Wolff potential defined by $(1.7)$. Finally, we prove the local Hölder continuity for supersolutions.

\section{Preliminaries}
We briefly introduce our assumptions. More information about, generalized $\Phi$-functions and Musielak-Orlicz-Sobolev spaces, can be found in J. Musielak monograph $\cite{ref25}$ and P. Harjulehto, P. H\"ast\"o monograph $\cite{ref9}$. We denote,  $\Omega$  a bounded domain of  $\R^n$ with $n \geq 2$,  $L^0(\Omega)$ the set of measurable functions on $\Omega$. $C$ is a generic constant whose value may change between appearances.

\begin{df}
A function $G: \Omega \times [0,\infty) \rightarrow [0,\infty]$ is called a generalized $\Phi$-function, denoted by $G(\cdot) \in \Phi(\Omega)$, if the following conditions hold
\begin{itemize}
\item[•] For each $t \in [0,\infty)$, the function $G(\cdot,t)$ is measurable.
\item[•] For a.e $x \in \Omega$, the function $G(x, \cdot)$ is an $\Phi$-function, i.e.
\begin{enumerate}
\item $G(x,0) =  \displaystyle \lim\limits_{t \rightarrow 0^+} G(x,t) = 0$ and $\displaystyle \lim\limits_{t \rightarrow \infty} G(x,t) = \infty$;
\item $G(x,\cdot)$ is increasing and convex.
\end{enumerate}
\end{itemize}
\end{df}
Note that, a generalized $\Phi$-function can be represented as
$$ G(x,t) =  \displaystyle \int_{0}^t g(x,s) \, \mathrm{d}s,$$
where $g(x,\cdot)$ is the right-hand derivative of $G(x,\cdot)$. Furthermore for each $x \in \Omega$, the function $g(x,\cdot)$ is right-continuous and nondecreasing, so we have the following inequality
\begin{equation}
g(x,a)b \leq g(x,a)a + g(x,b)b \; \; \text{for} \, x \in \Omega \; \text{and} \; a,b \geq 0
\end{equation}
We denote $G^+_{B}(t) := \sup_B G(x,t), \; G^-_B(t) := \inf_B G(x,t)$.  We say that  $G(\cdot)$ satisfies\\
$ (SC)$ If there exist two constants  $g_0,g^0 > 1$ such that,
$$ 1 < g_0 \leq \displaystyle \frac{t g(x,t)}{G(x,t)} \leq g^0.$$
$(A_0)$ If there exists a constant  $c_0 > 1$ such that,
$$ \displaystyle \frac{1}{c_0} \leq G(x,1) \leq c_0, \, \; \text{a.e} \; \, x \in \Omega.$$
$(A_{1})$ If there exists a positive constant $C$ such that,for every Ball $B_R$ with $R < 1$ and $ x,y \in B_R \cap \Omega$, we have
$$ G(x,t) \leq CG(y,t) \; \; \; \text{when} \; \; G^-_{B}(t) \in \left[1 , \frac{1}{R^n}\right].$$
$(A_{1,s})$ If there exists a positive constant $C$ such that, for every Ball $B_R$ with $R < 1$ and $ x,y \in B_R \cap \Omega$, we have
$$ G(x,t) \leq CG(y,t) \; \; \; \text{when} \; \; t^s \in \left[1 , \frac{1}{R^n}\right].$$

\noindent  Under the structure condition $(SC)$, we have the following inequalities
\begin{equation}
\sigma^{g_0} G(x,t) \leq G(x,\sigma t) \leq \sigma^{g^0} G(x,t), \; \;\text{for} \; x \in \Omega, \; \, t \geq 0 \; \text{and} \;\sigma \geq 1.
\end{equation}
\begin{equation}
\sigma^{g^0} G(x,t) \leq G(x,\sigma t) \leq \sigma^{g_0} G(x,t), \; \;\text{for} \; x \in \Omega, \; \, t \geq 0 \; \text{and} \;\sigma \leq 1.
\end{equation}
We define $G^*(\cdot)$ the conjugate $\Phi$-function of $G(\cdot)$, by
$$ G^*(x,s) := \sup_{t \geq 0}(st - G(x,t)), \; \, \; \text{for} \; x \in \Omega \; \text{and} \; s \geq 0.$$ 
Note that $G^*(\cdot)$ is also a generalized $\Phi$-function and can be represented as
$$ G^*(x,t) =  \displaystyle \int_{0}^t g^{-1}(x,s) \, \mathrm{d}s,$$
with $g^{-1}(x,s) : = \sup \{ t \geq 0 \; : \; g(x,t) \leq s \}$.\\
By this definition, $G(\cdot)$ and $G^*(\cdot)$ satisfy the following Young inequality 
$$ st \leq G(x,t) + G^*(x,s), \; \, \text{for} \; x \in \Omega \; \text{and} \; s,t \geq 0.$$
Furthermore, we have the equality if $s = g(x,t)$ or  $t = g^{-1}(x,s)$. So, if $G(\cdot)$ satisfies the condition $(SC)$, we have the following inequality
\begin{equation}
 G^*(x,g(x,t)) \leq (g^0-1)G(x,t), \; \forall x \in \Omega, t \geq 0.
\end{equation}

\begin{rma} If $G(\cdot) \in \Phi(\Omega)$ satisfies $(SC)$, then $G^*(\cdot)$ satisfies the structure condition: 
$$\displaystyle \frac{g^0}{g^0 - 1} \leq \displaystyle \frac{t g^{-1}(x,t)}{G^*(x,t)} \leq \frac{g_0}{g_0 - 1}.$$
\end{rma}

\begin{df}
We define the generalized Orlicz space, also called Musielak-Orlicz space, by
$$ L^{G(\cdot)}(\Omega) := \{ u \in L^0(\Omega) \; : \; \displaystyle \lim\limits_{\lambda \rightarrow 0}  \rho_{G(\cdot)}(\lambda|u|) = 0 \},$$
where $\rho_{G(\cdot)}(t) = \displaystyle \int_{\Omega} G(x,t) \, \mathrm{d}x$. If $G(\cdot)$ satisfies $(SC)$, then
$$ L^{G(\cdot)}(\Omega) = \{ u \in L^0(\Omega) \; : \; \rho_{G(\cdot)}(|u|) < \infty \}.$$
\end{df}

\noindent On the generalized Orlicz space, we define the following  norms\\
- Luxembourg norm $ \norme{u}_{G(\cdot)} = \inf \{ \lambda > 0 \; : \; \rho_{G(\cdot)}(\displaystyle \frac{u}{\lambda}) \leq 1 \}.$\\
- Orlicz norm $\norme{u}_{G(\cdot)}^0 = \sup \{ |\displaystyle \int_{\Omega} u(x)v(x) \, \mathrm{d}x| \; : \; v \in L^{G^*(\cdot)}(\Omega), \; \rho_{G^*(\cdot)}(v) \leq 1 \}.$\\
These norms are equivalent. Precisely, we have
$$\norme{u}_{G(\cdot)} \leq \norme{u}_{G(\cdot)}^0 \leq 2 \norme{u}_{G(\cdot)}.$$
Furthermore, by definition of Orlicz norm and Young inequality, we have
\begin{equation}
\norme{u}_{G(\cdot)} \leq \norme{u}_{G(\cdot)}^0 \leq \displaystyle \int_{\Omega} G(x,|u|) \, \mathrm{d}x + 1.
\end{equation}
\noindent The following proposition establishes properties of convergent sequences in generalized Orlicz spaces.

\begin{pro}
Let $G(\cdot) \in \Phi(\Omega)$ satisfies $(SC)$. For any sequence $(u_i) \in L^{G(\cdot)}(\Omega)$, we have the following properties
\begin{enumerate}
\item Fatou lemma: If $u_i \rightarrow u$ almost everywhere, then 
$$\displaystyle \int_{\Omega} G(x,|u(x)|) \, \mathrm{d}x \leq  \varliminf_{i \rightarrow \infty} \int_{\Omega} G(x,|u_i(x)|) \, \mathrm{d}x.$$
\item $\norme{u_i}_{G(\cdot)} \rightarrow 0  \; \; (\text{resp.} 1 ;\infty) \Longleftrightarrow \displaystyle \int_{\Omega} G(x,|u_i(x)|) \, \mathrm{d}x \rightarrow 0  \; \; (\text{resp.} 1 ;\infty).$
\item The functions $G(\cdot)$ and $G^*(\cdot)$ satisfy the H\"older inequality
$$ \left| \displaystyle \int_{\Omega} u(x)v(x) \, \mathrm{d}x \right| \leq 2\norme{u}_{G(\cdot)} \norme{v}_{G^*(\cdot)}, \; \; \text{for} \; u \in L^{G(\cdot)}(\Omega) \; \text{and} \; v \in L^{G^*(\cdot)}(\Omega).$$
\end{enumerate}
\end{pro}

\noindent The relation between a modular and its norm, under the structure condition $(SC)$, is given by the following proposition.

\begin{pro} Let $G(\cdot) \in \Phi(\Omega) $ satisfies $(SC)$. Then the following relations hold true
\begin{enumerate}
\item $\norme{u}_{G(\cdot)}^{g_0} \leq \rho_{G(\cdot)}(u) \leq \norme{u}_{G(\cdot)}^{g^0}, \; \forall u \in L^{G(\cdot)}(\Omega) \; \text{with} \; \norme{u}_{G(\cdot)} \geq 1.$
\item $\norme{u}_{G(\cdot)}^{g^0} \leq \rho_{G(\cdot)}(u) \leq \norme{u}_{G(\cdot)}^{g_0}, \; \forall u \in L^{G(\cdot)}(\Omega) \; \text{with} \; \norme{u}_{G(\cdot)} \leq 1.$
\end{enumerate}
\end{pro}

\begin{df}
We define the generalized Orlicz-Sobolev space by
$$ W^{1,G(\cdot)}(\Omega) := \{ u \in L^{G(\cdot)}(\Omega) \; : \; |\nabla u| \in L^{G(\cdot)}(\Omega), \; \, \text{in the distribution sense} \},$$
equipped with the norm
$$ \norme{u}_{1,G(\cdot)} = \norme{u}_{G(\cdot)} + \norme{\nabla u}_{G(\cdot)}.$$
\end{df}

\begin{rma}
The proposition $2.5$ remains true for the norm of the generalized Sobolev-Orlicz spaces.
\end{rma}

\begin{df}
 $ W^{1,G(\cdot)}_0(\Omega)$ is the closure of $C^\infty_0(\Omega)$ in $W^{1,G(\cdot)}(\Omega).$
\end{df}

Note that, if  $G(\cdot) \in \Phi(\Omega)$ satisfies the condition $(SC)$ and $(A_0)$, then $ W^{1,G(\cdot)}(\Omega)$ is a Banach, separable and reflexive space.

\section{Quasilinear elliptic equations}
Let $\mathcal{A} : \Omega\times\R^n \rightarrow \R^n$ be a function satisfying the following assumptions:
\begin{enumerate}
\item[] $a_1)$ $(x,\xi) \rightarrow \mathcal{A}(x,\xi)$ is a Carathéodory function.
\item[] $a_2)$ There exists a positive constant $c_1$ such that $$\mathcal{A}(x,\xi) \cdot \xi \geq c_1 g(x,|\xi|)|\xi|, \; \, \text{for} \; x \in \Omega \; \text{and} \; \xi \in \R^n.$$
\item[] $a_3)$ There exists a positive constant $c_2$ such that $$|\mathcal{A}(x,\xi)| \leq c_2g(x,|\xi|), \; \, \text{for} \; x \in \Omega \; \text{and} \; \xi \in \R^n.$$
\item[] $a_4)$ $(\mathcal{A}(x,\xi_1)- \mathcal{A}(x,\xi_2))\cdot(\xi_1-\xi_2) > 0, \; \, \text{for} \; x \in \Omega \; \text{and} \; \xi_1, \xi_2 \in \R^n$ with $\xi_1 \neq \xi_2.$
 \end{enumerate}

Under the previous conditions, we consider the following quasilinear elliptic equation.
\begin{equation}
 - \text{div} \mathcal{A}(x,\nabla u) = 0.
\end{equation}

\begin{df}
A function $u \in W^{1,G(\cdot)}(\Omega)$ is a solution to equation $(3.1)$ in $\Omega$ if
$$ \int_{\Omega} \mathcal{A}(x,\nabla u) \cdot \nabla \varphi \, \mathrm{d}x = 0$$
whenever  $\varphi \in C_0^\infty(\Omega)$.
\end{df}

\begin{df}
A function $u \in W^{1,G(\cdot)}(\Omega)$ is a supersolution (resp, subsolution) to equation $(3.1)$ in $\Omega$ if
$$ \int_{\Omega} \mathcal{A}(x,\nabla u) \cdot \nabla \varphi \, \mathrm{d}x \geq 0 \; \; (\text{resp,} \leq 0),$$
whenever  $\varphi \in C_0^\infty(\Omega)$ nonnegative.
\end{df}

By Giorgi–Nash-Moser theory for solutions to equation $(3.1)$ (see $\cite{ref2,ref3,ref11}$), we have the following Harnack estimates.

\begin{lem}
Let $G(\cdot) \in \Phi(\Omega)$ satisfies $(SC)$, $(A_0)$ and $(A_{1,n})$. If $u \in L^\infty(B_{2R})$ is a subsolution to equation $(3.1)$ in $B_{2R}$, then for any $q > 0$, there is a positive constant $C = C(q,c_1,c_2,g_0,g^0,\beta,n,\norme{u}_{\infty, B})$ such that
$$\esssup_{B} u^+ \; \leq C\left(\displaystyle \dashint_{B_{2R}} \overline{u}^q \, \mathrm{d}x \right)^{\frac{1}{q}},$$
with $\overline{u} = u^+ + R$.
\end{lem}

In $\cite{ref2}$, we have the more general version of the weak Harnack inequality for unbounded supersolutions.

\begin{lem}
Let $G(\cdot) \in \Phi(\Omega)$ satisfies $(SC)$ and $(A_0)$ and $u$ be a nonnegative supersolution to $(1.1)$ in $B_{2R}$. Assume $u$ that one of the following holds
\begin{enumerate}
\item[(1)] $G(\cdot)$ satisfies $(A_{1,s_*})$ and $\norme{u}_{L^s(B_{2R})} \leq d$, where $s_*= \frac{ns}{n+s}$ and  $s \in [g^0-g_0 , \infty]$.
\item[(2)] $G(\cdot)$ satisfies $(A_{1})$ and $\norme{u}_{W^{1,G(\cdot)}(B_{2R})} \leq d$.
\end{enumerate}
Then there exist positive constants $\gamma$ and $C$ such that
$$ \left(\displaystyle \dashint_{B_{2R}} (u+R)^\gamma \, \mathrm{d}x \right)^{\frac{1}{\gamma}} \leq C\essinf_{B_R}(u+R)$$
If $(1)$ holds with $s > \max\{\frac{n}{g_0},1\}(g^0-g_0)$ or if $(2)$ holds with $\frac{ng_0}{n-g_0} > g^0$, then the weak Harnack inequality holds for any $\gamma < \gamma_0$ with
$$\gamma_0 := 
\begin{cases}
\frac{n(g_0-1)}{n-g_0} & \text{if} \; g_0 < n \\
\infty & \text{if} \; g_0 \geq n
\end{cases}$$
\end{lem}

As in $\cite{ref3}$ or we take $s = \infty$ in the previous lemma, if $u$ is locally bounded we can replace condition $(1)$ by the condition $(A_{1,n})$.

\begin{lem}
Let $G(\cdot) \in \Phi(\Omega)$ satisfies $(SC)$, $(A_0)$ and $(A_{1,n})$. If $u \in L^\infty(B_{2R})$ is a supersolution to equation $(3.1)$ in $B_{2R}$, then there exist positive constants $\gamma$ and $C$ such that
$$ \left(\displaystyle \dashint_{B_{2R}} (u+R)^\gamma \, \mathrm{d}x \right)^{\frac{1}{\gamma}} \leq C\essinf_{B_R}(u+R)$$
\end{lem}

Under the previous conditions $(a_1)$, $(a_2)$, $(a_3)$ and $(a_4)$, we consider the following quasilinear elliptic equation with data measure.
\begin{equation}
 - \text{div} \mathcal{A}(x,\nabla u) = \mu.
\end{equation}

\begin{df}
A function $u \in W^{1,G(\cdot)}(\Omega)$ is a solution of the equation $(3.2)$ in $\Omega$ if
$$ \int_{\Omega} \mathcal{A}(x,\nabla u) \cdot \nabla \varphi \, \mathrm{d}x = \int_{\Omega}  \varphi \, \mathrm{d}\mu,$$
whenever  $\varphi \in C_0^\infty(\Omega).$
\end{df}

From the density of $C_0^\infty(\Omega)$, the class of test functions can be extended to $W^{1,G(\cdot)}_0(\Omega)$  in $(3.2)$.

\begin{lem}
If  $u \in W^{1,G(\cdot)}(\Omega)$ is a solution of the equation $(3.2)$ in  $\Omega$, then
$$ \int_{\Omega} \mathcal{A}(x,\nabla u) \cdot \nabla \varphi \, \mathrm{d}x  = \int_{\Omega}  \varphi \, \mathrm{d}\mu,$$
whenever  $\varphi \in W^{1,G(\cdot)}_0(\Omega)$.
\end{lem}

By the monotone condition of $\mathcal{A}$, we have the following weak comparison principle $\cite{ref13}$.

\begin{lem}
Let $u,v \in W^{1,G(\cdot)}(\Omega)$. If
$$ \int_{\Omega} \mathcal{A}(x,\nabla u) \cdot \nabla \varphi \, \mathrm{d}x  \leq  \int_{\Omega} \mathcal{A}(x,\nabla v) \cdot \nabla \varphi \, \mathrm{d}x,$$
for all nonnegative $\varphi \in W^{1,G(\cdot)}_0(\Omega)$ and $(u-v)^+ \in W^{1,G(\cdot)}_0(\Omega)$, then $u \leq v$, a.e in $\Omega$.
\end{lem}

\section{Existence of solution}

After a preliminary list of lemmas, we use the monotone operator's theory to prove the existence of solutions of the Dirichlet problem to equation $(3.2)$ with Sobolev boundary values.\\
The next lemmas are proved in $\cite{ref4}$.

\begin{lem}
Let $G(\cdot) \in \Phi(\Omega)$ satisfies $(SC)$ and $(A_0)$. Let $(u_i)$ be a sequence in $L^{G(\cdot)}(\Omega)$. If $u_i \rightarrow u$ in $L^{G(\cdot)}(\Omega)$, then there exists a subsequence $(u_{i_j})$ of $(u_i)$ which converge to $u$, a.e in $\Omega$.
\end{lem}

\begin{lem}
Let $G(\cdot) \in \Phi(\Omega)$ satisfies $(SC)$ and $(A_0)$. Let $(u_i)$ be a bounded sequence in $L^{G(\cdot)}(\Omega)$. If $u_i \rightarrow u$, a.e in $\Omega$, then $u_i$ converge to $u$ weakly in $L^{G(\cdot)}(\Omega)$.
\end{lem}

\begin{Th}
Let $G(\cdot) \in \Phi(\Omega)$ satisfies $(SC)$ and $(A_0)$. Let $\mu$ be a signed Radon measure in $\left(W_0^{1,G(\cdot)}\Omega)\right)^*$ and $\theta \in W^{1,G(\cdot)}(\Omega)$. Then there exists $u \in W^{1,G(\cdot)}(\Omega)$ be a solution of equation $(3.2)$ such that $u - \theta \in W^{1,G(\cdot)}_0(\Omega)$.
\end{Th}

\begin{proof}
Define a mapping $T: W^{1,G(\cdot)}_0(\Omega) \rightarrow (W^{1,G(\cdot)}_0(\Omega))^*$ such that, for $w \in W^{1,G(\cdot)}_0(\Omega)$
$$\langle Tw, \varphi \rangle := \displaystyle \int_{\Omega} \mathcal{A}(x,\nabla (w+\theta)) \cdot \nabla \varphi \, \mathrm{d}x.$$
The mapping $T$ is well defined, indeed by the condition $(a_3)$,
$$| \langle Tw, \varphi \rangle | \leq  \displaystyle c_2\int_{\Omega} g(x,|\nabla (w + \theta)|)|\nabla \varphi| \, \mathrm{d}x.$$
Using the H\"older inequality and inequalities $2.5$ and $2.4$, we have
\begin{equation}
\begin{array}{ll}
| \langle Tw , \varphi \rangle | & \leq \displaystyle 2c_2\norme{g(x,|\nabla (w + \theta)|)}_{G^*(\cdot)} \norme{\nabla \varphi}_{G(\cdot)}\\
& \leq  \displaystyle 2c_2 \left(\int_{\Omega} G^*(x,g(x,|\nabla (w+\theta)|)) \, \mathrm{d}x + 1\right)\norme{\nabla \varphi}_{G(\cdot)}\\
& \leq \displaystyle 2c_2 \left((g^0-1)\int_{\Omega} G(x,|\nabla (w+\theta)|) \, \mathrm{d}x + 1 \right)\norme{\nabla \varphi}_{G(\cdot)}.
\end{array}
\end{equation}
We apply the general result $\cite{ref19}$ which asserts that if $T$ is a bounded, coercive, demicontinuous map, then for all $\mu \in (W^{1,G(\cdot)}_0(\Omega))^*$ the equation $Tw = \mu$ has a solution $w \in W^{1,G(\cdot)}_0(\Omega)$.\\
\textbf{-i)} By the inequality $(4.1)$, the map $T$ is bounded.\\
\textbf{-ii)} Next, we show that $T$ is coercive. For any $w \in W^{1,G(\cdot)}_0(\Omega)$ and $\varphi \in W^{1,G(\cdot)}_0(\Omega)$, by the condition $(a_2)$ and, the condition $(SC)$, we have
$$\begin{array}{ll}
\langle Tw, w \rangle & = \displaystyle \int_{\Omega} (\mathcal{A}(x,\nabla (w+\theta)) \cdot \nabla w \, \mathrm{d}x\\
& \geq \displaystyle c_1\int_{\Omega} g(x,|\nabla (w+\theta)|) |\nabla w| \, \mathrm{d}x\\
& \geq \displaystyle c_1\int_{\Omega} g(x,|\nabla (w+\theta)|)|\nabla (w+\theta)| \, \mathrm{d}x - \displaystyle c_1\int_{\Omega} g(x,|\nabla (w+\theta)|)|\nabla \theta| \, \mathrm{d}x\\
& \geq \displaystyle c_1g_0\int_{\Omega} G(x,|\nabla (w+\theta)|) \, \mathrm{d}x - \displaystyle c_1\int_{\Omega} g(x,|\nabla (w+\theta)|)|\nabla \theta| \, \mathrm{d}x.
\end{array}$$
Using inequality $(2.1)$, for $a=|\nabla (w+\theta)|$ and $b=g^0|\nabla \theta|$, and the condition $(SC)$, we get
$$\begin{array}{ll}
\displaystyle \int_{\Omega} g(x,|\nabla (w+\theta)|)|\nabla \theta| \, \mathrm{d}x & \leq \displaystyle \frac{1}{g^0} \int_{\Omega} g(x,|\nabla (w+\theta)|)|\nabla (w+\theta)| \, \mathrm{d}x \\
& \quad + \displaystyle \frac{1}{g^0} \int_{\Omega} g(x,g^0|\nabla \theta|)g^0|\nabla \theta| \, \mathrm{d}x\\
& \leq \displaystyle \int_{\Omega} G(x,|\nabla (w+\theta)|) \, \mathrm{d}x + \displaystyle  (g^0)^{g^0}\int_{\Omega} G(x,|\nabla \theta|) \, \mathrm{d}x.
\end{array}$$
Hence,
$$\langle Tw, w \rangle \geq \displaystyle c_1(g_0-1)\int_{\Omega} G(x,|\nabla (w+\theta)|) \, \mathrm{d}x - \displaystyle c_1 (g^0)^{g^0}\int_{\Omega} G(x,|\nabla \theta|) \, \mathrm{d}x.$$
Choosing $w$ sufficiently large, we can assume that $\norme{\nabla \theta}_{G(\cdot)} + \displaystyle \frac{1}{2} \leq \displaystyle\frac{1}{2}\norme{\nabla w}_{G(\cdot)}$. Then
$$1 \leq \norme{\nabla w}_{G(\cdot)} \leq \norme{\nabla (w+\theta)}_{G(\cdot)} + \norme{\nabla \theta}_{G(\cdot)} \leq \norme{\nabla (w+\theta)}_{G(\cdot)} + \frac{1}{2}\norme{\nabla w}_{G(\cdot)}.$$
So,
$$ 1 \leq \norme{\nabla w}_{G(\cdot)} \leq 2\norme{\nabla (w+\theta)}_{G(\cdot)}.$$
Hence, by Proposition $2.5$, we have
$$\frac{\langle Tw, w \rangle}{\norme{\nabla w}_{G(\cdot)}} \geq \frac{c_1(g_0-1)}{2^{g_0}}\norme{\nabla w}_{G(\cdot)}^{g_0-1} - \frac{c_1 (g^0)^{g^0}}{\norme{\nabla w}_{G(\cdot)}}\int_{\Omega} G(x,|\nabla \theta|) \, \mathrm{d}x.$$
The right hand side goes to $+\infty$ as $\norme{\nabla w}_{G(\cdot)}\rightarrow \infty$. Hence, $T$ is coercive.\\
\textbf{-iii)} Now we show $T$ is demicontinuous. In fact, let  $w_i$ be a sequence that converges to an element $w$ in $W^{1,G(\cdot)}_0(\Omega)$. By Lemma $4.1$, there exists a subsequence $ w_{i_j}$ of $w_i$, such that $w_{i_j} \rightarrow w \,$, a.e. in $\Omega$.\\
Since the mapping $\xi \mapsto \mathcal{A}(x,\xi)$ is continuous, then
$$ \mathcal{A}(x, \nabla w_{i_{j}})  \rightarrow \mathcal{A}(x,\nabla w), \, \; \text{a.e. in} \; \Omega.$$
Or by the condition $(a_3)$, Remark $2.2$ and inequalities $(2.2)$, $(2.4)$, we have
$$\begin{array}{ll}
\displaystyle \int_{\Omega} G^*(x, |\mathcal{A}(x,\nabla w_{i_j})|) \, \mathrm{d}x
& \leq \displaystyle \int_{\Omega} G^*(x, c_2 g(x,|\nabla w_{i_j}|) \, \mathrm{d}x\\
& \leq \text{max}\left( c_2 , (c_2)^{\frac{g_0}{g_0 - 1}} \right) \displaystyle \int_{\Omega} G^*(x, g(x,|\nabla w_{i_j}|) \, \mathrm{d}x\\
& \leq (g^0 -1) \text{max}\left( c_2 , (c_2)^{\frac{g_0}{g_0 - 1}} \right) \displaystyle \int_{\Omega} G(x,|\nabla w_{i_j}|) \, \mathrm{d}x.
\end{array}$$
Hence, from the inequality $(2.5)$, the $L^{G^*(\cdot)}(\Omega)-$norms of $ \mathcal{A}(x, \nabla w_{i_j})$ is uniformly bounded. So, by Lemma $4.2$, we have
$$\mathcal{A}(x,\nabla w_{i_j}) \rightharpoonup \mathcal{A}(x,\nabla w),$$
weakly in $L^{G^*(\cdot)}(\Omega)$. Since the weak limit is independent of the choice of the subsequence, it follows that
$$\mathcal{A}(x,\nabla w_{i}) \rightharpoonup \mathcal{A}(x,\nabla w),$$
weakly in $L^{G^*(\cdot)}(\Omega)$. Consequently, for all $\varphi \in W^{1,G(\cdot)}_0(\Omega)$.
$$ \langle T w_i , \varphi \rangle \rightarrow \langle T w , \varphi \rangle.$$
Hence, $T$ is demicontinuous on $W^{1,G(\cdot)}_0(\Omega)$.\\
Therefore, $u = w + \theta$ is a solution to equation $(3.2)$ 
\end{proof}

\section{Wolff potential bounded}

In this Section, we proof pointwise potential bounds for solutions. First, we introduce the Wolff potential in the generalized Orlicz setting.

\begin{df}
Let $\mu$ be a nonneagative Radon measure on $\R^n$ and $R > 0$. We define the Wolff potential of $\mu$ and $G(\cdot)$ by
$$ W^\mu_{G(\cdot)}(x,R) := \displaystyle \int_0^R g^{-1}\left( x,\frac{\mu(B(x,s)}{s^{n-1}}\right) \, \mathrm{d}s.$$
\end{df}

\begin{exs} 
In the variable exponent case, $G(x,t) = \displaystyle \frac{t^{p(x)}}{p(x)}$, (see $\cite{ref1,ref20}$) 
$$ W^\mu_{p(\cdot)}(x,R) = \displaystyle \int_0^R \left(\frac{\mu(B(x,s)}{s^{n-p(x)}}\right)^\frac{1}{p(x)-1} \, \frac{\mathrm{d}s}{s} = \displaystyle \int_0^R \left(\frac{\mu(B(x,s)}{s^{n-1}}\right)^\frac{1}{p(x)-1} \, \mathrm{d}s.$$
In the Orlicz case, $G(x,t) = G(t)$, (see  $\cite{ref21}$)
$$ W^\mu_{G}(x,R) = \displaystyle \int_0^R g^{-1}\left(\frac{\mu(B(x,s)}{s^{n-1}}\right) \, \mathrm{d}s.$$
\end{exs}

\noindent The following lemma establishes that the functions $ G^- $ is an $ \Phi$-function $\cite{ref3}$.

\begin{lem}
If $G(\cdot) \in \Phi(\Omega)$ satisfies $(SC)$ and $(A_0)$, then $G^-$ is an $\Phi$-function and $ 1 < g_0 \leq \displaystyle \frac{t \tilde{g}(t)}{G^-(t)} \leq g^0$,  where $\tilde{g}$ is the right-hand derivative of $G^-$.
\end{lem}

The following lemma gives a more flexible characterization of $(A_{1,n}) \; \cite{ref9}$.

\begin{lem}
Let $\Omega \subset \R^n$ be convex, $G(\cdot) \in \Phi(\Omega)$ and $0 < r \leq s$. Then $G(\cdot)$ satisfies $(A_{1,n})$ if, and only if, there exists $\beta > 0$ such that, for every $ x,y \in B_R \subset \Omega$ with $|B_R| \leq 1$, we have
$$ G_B(x,\beta t) \leq G_B(y,t) \; \; \; \text{when} \; \; t \in \left[ r , \frac{s}{R}\right]. $$
\end{lem}

\subsection{Potential lower bounded}
The following Lemma gives the Caccioppoli type estimate of supersolution to equation $(3.1)$ (see $\cite{ref3}$).

\begin{lem}
Let $u$ be a nonpositive supersolution of $(3.1)$ in a ball $2B$, $\eta \in C^\infty_0(2B)$ with $0 \leq \eta \leq 1$ and $|\nabla \eta| \leq \frac{1}{R} $. Then, there exits a constant $C$ such that
$$\displaystyle \int_{\frac{3}{2} B} G(x, |\nabla u|) \eta^{g^0} \, \mathrm{d}x \leq C \displaystyle \int_{\frac{3}{2} B} G^+(\frac{-u}{R}) \, \mathrm{d}x.$$
\end{lem}

\begin{Th}
Let $G(\cdot) \in \Phi(\Omega)$ satisfies $(SC)$, $(A_0)$ and $(A_{1,n})$. Let $u$ be a nonnegative weak solution to $(3.2)$ with nonnegative Radon measure $\mu$ in $\Omega$ and $B_R = B(x_0, R) \subset B_{2R} \Subset \Omega$. If $u$ is lower semicontinuous at $x_0$, then there exists a constant $C = C\left(c_1,c_2,g_0,g^0,n,\frac{\norme{u}_{g_0,B}}{|B|}\right) > 0$ such that
$$u(x_0) \geq C W^{\mu}_{G(\cdot)}(x_0,R) + \inf_{2B} u - 2R.$$
\end{Th}

\begin{proof}
We set $a = \inf_{2B} u$ and, $b = \inf_{B} u \; , \; v = \min\{u,b\} - a + R\; , \; u_j = \min \{u,j\}$. Choose $\omega = v\eta^{g^0} $ such that $\eta \in C^\infty_0(\frac{3}{2}B)$ with $0 \leq \eta \leq 1$, and $|\nabla \eta| \leq \displaystyle \frac{C}{R} $, we have
$$\begin{array}{ll}
(b-a+R)\mu(B) & \leq  \displaystyle\int_{\frac{3}{2} B}  \omega \, \mathrm{d}\mu \\
& =  \displaystyle \int_{\frac{3}{2} B} \mathcal{A}(x,\nabla u) \cdot \nabla \omega \, \mathrm{d}x \\
&\leq  \displaystyle \int_{\frac{3}{2} B} (\mathcal{A}(x,\nabla u) \cdot \nabla v)\eta^{g^0} \, \mathrm{d}x + \displaystyle \int_{\frac{3}{2} B} (\mathcal{A}(x,\nabla u) \cdot \nabla \eta)\eta^{g^0-1}v \, \mathrm{d}x.\\
\end{array}$$
By the conditions $(a_3)$ and $(SC)$, we have
$$\begin{array}{ll}
I_1 & := \displaystyle \int_{\frac{3}{2} B} (\mathcal{A}(x,\nabla u) \cdot \nabla v)\eta^{g^0} \, \mathrm{d}x\\
& \leq c_2 g^0 \displaystyle \int_{\frac{3}{2} B} G(x, |\nabla v)|) \eta^{g^0} \, \mathrm{d}x.
\end{array}$$
By the conditions $(a_3)$, we have
$$\begin{array}{ll}
I_2 & := \displaystyle \int_{\frac{3}{2} B} (\mathcal{A}(x,\nabla u) \cdot \nabla \eta)\eta^{g^0-1}v \, \mathrm{d}x\\
& \leq c_2\displaystyle \int_{\frac{3}{2} B} g(x, |\nabla v|)|\nabla \eta|\eta^{g^0-1}v \, \mathrm{d}x.
\end{array}$$
As $v \leq b - a + R$ and $|\nabla \eta| < \displaystyle \frac{C}{R}$, we have
$$I_2 \leq C \displaystyle \frac{b-a+R}{R} \displaystyle \int_{\frac{3}{2} B} g(x, |\nabla v|)\eta^{g^0-1} \, \mathrm{d}x.$$
Using inequality $(2.1)$ and the condition $(SC)$, we get
$$I_2 \leq C \left( \displaystyle \int_{\frac{3}{2} B} G(x, |\nabla v|)\eta^{g^0} \, \mathrm{d}x + \displaystyle \int_{\frac{3}{2} B} G\left( x, \frac{b-a+R}{R}\right) \, \mathrm{d}x \right).$$
Collecting the previous estimations of $I_1$ and $I_2$, we obtain
$$(b-a+R)\mu(B) \leq C\left( \displaystyle \int_{\frac{3}{2} B} G(x, |\nabla v|)\eta^{g^0} \, \mathrm{d}x + \displaystyle \int_{\frac{3}{2} B} G\left( x, \frac{b-a+R}{R}\right) \, \mathrm{d}x \right).$$
Or, by Lemma $5.5$, we have
$$\displaystyle \int_{\frac{3}{2} B} G(x,|\nabla( v-(b-a+R))|)\eta^{g^0} \, \mathrm{d}x \leq C \displaystyle \int_{\frac{3}{2} B} G^+\left(\frac{b-a+R-v}{R}\right) \, \mathrm{d}x.$$
Hence,
$$ (b-a+R)\mu(B)\leq C \displaystyle \int_{\frac{3}{2} B} G^+\left(\frac{b-a+R}{R}\right) \, \mathrm{d}x.$$
Since, $L^{G(\cdot)}(B) \subset L^{g_0}(B)$ (see $\cite{ref9}$), we have
$$1 \leq \displaystyle \frac{b-a+R}{R} \leq \displaystyle \frac{b + 1}{R} \leq \displaystyle \frac{\frac{\norme{u}_{g_0,B}}{|B|}+1}{R}.$$
Then by Lemma  $5.4$, there exists a constant $C>0$ depend of $\displaystyle\frac{\norme{u}_{g_0,B}}{|B|}$ such that
$$G^+\left(\frac{b-a+R}{R}\right) \leq C G\left( x_0,\frac{b-a+R}{R}\right). $$
Hence,
$$ (b-a+R)\mu(B) \leq C R^{n} G\left( x_0,\frac{b-a+R}{R}\right).$$
So, by the condition $(SC)$, we have
$$ \mu(B) \leq C R^{n-1} g\left( x_0,\frac{b-a+R}{R}\right).$$
From Remark $2.2$, condition $(SC)$ and inequalities $(2.2)$, $(2.3)$, we have
\begin{equation}
C R g^{-1}\left( x_0,\frac{\mu(B)}{R^{n-1}} \right) \leq \inf_{B}u - \inf_{2B}u + R.
\end{equation}
Let $R_j := 2^{1-j}R$. Iterating inequality $(5.1)$, we get
$$\begin{array}{ll}
C \displaystyle \sum_{j=1}^{\infty} R_j g^{-1}\left( x_0,\frac{\mu(B_j)}{R_j^{n-1}}\right) & \leq \displaystyle \sum_{j=1}^{\infty} (\inf_{B_j} u - \inf_{B_{j-1}} u + R_j)\\
& \leq \displaystyle \lim_{k\rightarrow \infty} (\inf_{B_k} u) - \inf_{2B} u + \sum_{j=1}^{\infty} R_j.
\end{array}$$
As $u$ is lower semicontinuous at $x_0$, then
$$C \displaystyle \sum_{j=1}^{\infty} R_j g^{-1}\left( x_0,\frac{\mu(B_j)}{R_j^{n-1}}\right) \leq u(x_0) - \inf_{2B} u + 2R.$$
Or, by Remark $2.2$, condition $(SC)$ and inequalities $(2.2)$, $(2.3)$, we have
$$\int_{R_{j+1}}^{R_j}  \, g^{-1} \left(x_0,\frac{\mu(B(x_0,s)}{s^{n-1}}\right) \mathrm{d}s \leq C R_jg^{-1}\left( x_0,\frac{\mu(B_j)}{R_j^{n-1}}\right).$$
Hence,
$$W^{\mu}_{G(\cdot)}(x_0,R) \leq C \sum_{j=1}^{\infty} R_j g^{-1}\left( x_0,\frac{\mu(B_j)}{R_j}\right).$$
Therefore,
$$C W^{\mu}_{G(\cdot)}(x_0,R) + \inf_{2B} u - 2R \leq  u(x_0). $$
This gives the claim. 
\end{proof} 

\subsection{Potential upper bounded}
In order to give the potential upper bound estimate, we briefly recall some results on Lorentz spaces $\cite{ref7}$.

\begin{df}
Let $f \in L^0(\Omega)$. For $q > 0$, we define
$$\norme{f}^{q,\infty}_\Omega := \sup_{t > 0} t\left|\{x \in \Omega \; : \; |f(x)| \geq t\}\right|^\frac{1}{q},$$
$$\norme{f}^{q,1}_\Omega := q \int_0^\infty \left|\{x \in \Omega \; : \; |f(x)| \geq t\} \right|^\frac{1}{q} \, \mathrm{d}t.$$
By $L^{q,\infty}(\Omega)$ (resp., $L^{q,1}(\Omega)$), we denote the space of all measurable functions $f$ on $\Omega$ such that $\norme{f}^{q,\infty}_\Omega < \infty$ (resp., $\norme{f}^{q,1}_\Omega < \infty$ ). Such space are called Lorentz space.
\end{df}

\begin{pro}
We have the following properties
\begin{enumerate}
\item $ L^{q,1}(\Omega) \subset L^{q}(\Omega) \subset L^{q,\infty}(\Omega).$
\item For any nonnegative constant $l$,
$$\norme{f^+}^{q,\infty}_\Omega \leq l |\Omega|^\frac{1}{q} + \norme{(f-l)^+}^{q,\infty}_\Omega.$$
\item For $q > 1$ and $f,g \in L^0(\Omega)$, we have the H\"older type inequality
$$\left|\int_\Omega fg  \, \mathrm{d}x\right| \leq \norme{f}^{q^*,1}_\Omega\norme{g}^{q,\infty}_\Omega,$$
where $\frac{1}{q} + \frac{1}{q^*} = 1$.
\end{enumerate}
\end{pro}

\begin{lem}
Let $G(\cdot) \in \Phi(\Omega)$ satisfies $(SC)$ and $(A_0)$. If $(t_i)$ is a family in $\R^+$, then there exists a positive constant $C$ such that
$$ g(x,\sup_{i > 0} t_i) \leq C \sup_{i > 0} g(x,t_i).$$
\end{lem}

\begin{proof} First case $\sup_{i > 0} t_i < \infty $. For any $\epsilon < 1$, there exists $i_0  > 0$ such that
$$\sup_{i > 0} t_i < t_{i_0} + \epsilon .$$
Then, by the condition $(SC)$, the inequalities $(2.2)$, $(2.3)$ and $(A_0)$, we have
$$\begin{array}{ll}
g(x,\sup_{i > 0} t_i) & \leq g(x, t_{i_0} + \epsilon)\\
& \leq \displaystyle \frac{g^0}{g_0}2^{g^0-1}(g(x, t_{i_0}) +  g(x, \epsilon))\\
& \leq \; \, \displaystyle \frac{g^0}{g_0} 2^{g^0-1}(\sup_{i > 0} g(x,t_i)) + \frac{g^0}{g_0}\epsilon^{g_0-1} g(x, 1)).\\
\end{array}$$
Taking the limit $\epsilon \rightarrow 0$, we get
$$ g(x,\sup_{i > 0} t_i) \leq \displaystyle \frac{g^0}{g_0} 2^{g^0-1}\sup_{i > 0} g(x,t_i),$$
so the claim holds with $C=\displaystyle \frac{g^0}{g_0} 2^{g^0-1}$.\\
Second case $\sup_{i > 0} t_i = \infty $ the inequality is holds. 
\end{proof} 

\begin{Th}
Let $G(\cdot) \in \Phi(\Omega)$ satisfies $(SC)$, $(A_0)$ and $(A_{1,n})$. Let $u$ be a nonnegative weak solution to $(3.2)$ with nonnegative Radon measure $\mu$ in $\Omega$ and $B_R = B(x_0, R) \subset B_{2R} \Subset \Omega$. If $u \in L^{\chi^\prime,\infty}(\frac{1}{2}B)$ with $\chi^\prime = \frac{n(g_0-1)}{n-1}$ and lower semicontinuous at $x_0$ then for any $\gamma > 0$, there exists a positive constant $C = C\left( c_1,c_2,g_0,g^0,\gamma,n,\frac{\norme{u}^{\chi^\prime,\infty}_{\frac{1}{2}B}}{|\frac{1}{2}B|^\frac{1}{\chi^\prime}}\right)$ such that
$$u(x_0) \leq C\left( R+ \left(\dashint_{B\setminus\frac{1}{2}\overline{B}} u^\gamma \, \mathrm{d}x\right)^\frac{1}{\gamma} +  W^{\mu}_{G(\cdot)}(x_0,2R) \right).$$
\end{Th}
\begin{proof}
Let $u$ be a nonnegative solution to equation $(3.2)$ and $A = B \setminus \frac{1}{2}\overline{B}$. As in $\cite{ref4}$, there exists $v \in W^{1,G(\cdot)}(A)$ be the solution to equation
$$- \text{div} \mathcal{A}(x,\nabla v) = 0 \; \text{in} \; A,$$
such that $v-u \in W^{1,G(\cdot)}_0(A).$\\
\textbf{First step:} Fix $\Psi \in W^{1,G(\cdot)}_0(\frac{3}{4}B)$ such that $0 \leq \Psi \leq 1$. Extend $\Psi$ as $\Psi = 0$ outside of $\frac{3}{4}B$ and $v$ as $v=u$ outside of $A$, we show the following inequality
\begin{equation}
\int_{B} \mathcal{A}(x,\nabla v) \cdot \nabla \Psi \, \mathrm{d}x \leq 2\mu(\overline{B}).
\end{equation}
Indeed, from the definition of $u$ and $v$, we have
\begin{equation}
0 \leq \int_{A}  \varphi \, \mathrm{d}\mu =  \int_{A} (\mathcal{A}(x,\nabla u)- \mathcal{A}(x,\nabla v)) \cdot \nabla \varphi \, \mathrm{d}x.
\end{equation}
for any nonnegative $\varphi \in W^{1,G(\cdot)}_0(A)$. So, by Lemma $3.8$ and inequality $(5.3)$, we have $v \leq u$, a.e in $A$. Using the inequality $(5.3)$  with $\varphi = I_\epsilon(u-v)\Psi$ where $I_\epsilon(t) = \epsilon^{-1}\min\{t,\epsilon\}$, we obtain
$$\begin{array}{ll}
& \displaystyle \int_{A} \left((\mathcal{A}(x,\nabla v)- \mathcal{A}(x,\nabla u)) \cdot \nabla \Psi\right) I_\epsilon(u-v) \, \mathrm{d}x \\
& \leq \displaystyle \int_{A} ((\mathcal{A}(x,\nabla u)- \mathcal{A}(x,\nabla v)) \cdot \nabla I_\epsilon(u-v)) \Psi \, \mathrm{d}x\\
& \leq \displaystyle \int_{A} (\mathcal{A}(x,\nabla u)- \mathcal{A}(x,\nabla v)) \cdot \nabla I_\epsilon(u-v) \, \mathrm{d}x.
\end{array}$$
Again, we use the inequality $(5.3)$ with $\varphi = I_\epsilon(u-v)$, we get
$$\displaystyle \int_{A} (\mathcal{A}(x,\nabla u)- \mathcal{A}(x,\nabla v)) \cdot \nabla I_\epsilon(u-v) \, \mathrm{d}x \leq \displaystyle \int_{B}  I_\epsilon(u-v) \, \mathrm{d}\mu \leq \mu(\overline{B}).$$
Hence,
$$\displaystyle \int_{A} \left((\mathcal{A}(x,\nabla v)- \mathcal{A}(x,\nabla u)) \cdot \nabla \Psi\right) I_\epsilon(u-v) \, \mathrm{d}x \leq \mu(\overline{B}).$$
Take the limit $\epsilon \rightarrow 0$, we get
$$\int_{\{x \in A \, : \, u(x) > v(x)\}} (\mathcal{A}(x,\nabla v)- \mathcal{A}(x,\nabla u)) \cdot \nabla \Psi \, \mathrm{d}x  \leq \mu(\overline{B}).$$
Since, $\nabla u = \nabla v$ a.e on $\{x \in A \, : \, u(x) = v(x)\} $ and, $u \geq v$ a.e in $A$, $v=u$ outside of $A$, then
\begin{equation}
\int_{B} (\mathcal{A}(x,\nabla v)- \mathcal{A}(x,\nabla u)) \cdot \nabla \Psi \, \mathrm{d}x  \leq \mu(\overline{B}).
\end{equation}
On the other hand, as $\Psi \in W^{1,G(\cdot)}_0(B)$. Then, by the definition of solution $u$, we have
\begin{equation}
\int_{B} \mathcal{A}(x,\nabla u) \cdot \nabla \Psi \, \mathrm{d}x = \int_{B}  \Psi \, \mathrm{d}\mu \leq \mu(\overline{B}).
\end{equation}
Combining the two inequalities $(5.4)$ and $(5.5)$, we obtain the inequality $(5.2)$
$$ \int_{B} \mathcal{A}(x,\nabla v) \cdot \nabla \Psi \, \mathrm{d}x \leq 2\mu(\overline{B}).$$
\textbf{Second step:} We show the following inequality
\begin{equation}
\begin{array}{ll}
 & \displaystyle \frac{1}{|\frac{1}{2}B|^{\frac{1}{\chi^\prime}}}\norme{u}^{\chi^\prime,\infty}_{\frac{1}{2}B}\\
 & \leq C \left( \left(\displaystyle \dashint_{B \setminus\frac{1}{2}B} u^\gamma \, \mathrm{d}x\right)^\frac{1}{\gamma} + R g^{-1}\left( x_0,\displaystyle\frac{\mu(\overline{B})}{R^{n-1}}\right)+R \right),
 \end{array}
\end{equation}
where $\chi^\prime = \frac{n}{n-1}(g_0-1)$. In one hand, we have $v$ is a solution in $A$ to equation $(3.1)$, then, by $\cite{ref3}$, we have $v$ is locally bounded in $A$. So, by Lemma $3.3$
\begin{equation}
\esssup_{S} v^+  \leq C\left(\displaystyle \dashint_{A} \overline{v}^\gamma \, \mathrm{d}x\right)^\frac{1}{\gamma},
\end{equation}
where $S= \cup_{x \in \partial \frac{3}{4}B}B(x,\frac{R}{8})$ and $\overline{v}=v^+ + R$. Denoting $l = \esssup_{S} v^+$, then, by the conditions $(SC)$ and $(a_2)$, for any positive constant $k$, we have
$$\begin{array}{ll}
& c_1g_0 \displaystyle \int_{B} G\left( x,|\nabla \min\{(v^+ - l)^+,k\}|\right) \, \mathrm{d}x\\ 
& \leq c_1 \displaystyle \int_{B} g(x, |\nabla \min\{(v^+ - l)^+,k\}|) |\nabla \min\{(v^+ - l)^+,k\}| \, \mathrm{d}x\\
& \leq \displaystyle \int_{B} \mathcal{A}(x,\nabla \min\{(v^+ - l)^+,k\}) \cdot \nabla \min\{(v^+ - l)^+,k\} \, \mathrm{d}x\\
& \leq \displaystyle \int_{B} \mathcal{A}(x,\nabla v) \cdot \nabla \min\{(v^+ - l)^+,k\} \, \mathrm{d}x.
\end{array}$$
Note that $(v^+ - l)^+=0$, a.e on $S$, so $(v^+ - l)^+ \in W^{1,G(\cdot)}_0(\frac{3}{4}B)$. Then, by the inequality $(5.2)$ for $\Psi = k^{-1}\min\{(v^+ - l)^+,k\} \in W^{1,G(\cdot)}_0(\frac{3}{4}B)$, we get
$$\displaystyle \int_{B} \mathcal{A}(x,\nabla v) \cdot \nabla \min\{(v^+ - l)^+,k\} \, \mathrm{d}x \leq 2k\mu(\overline{B}).$$
Hence,
\begin{equation}
\int_{B} G\left( x,|\nabla \min\{(v^+ - l)^+,k\}|\right) \, \mathrm{d}x \leq Ck\mu(\overline{B}).
\end{equation}
In the other hand, by the Sobolev inequality for the function $G^-\left(\min \{ \displaystyle \frac{(v^+ - l)^+}{R} , \displaystyle \frac{k}{R} \} \right) \in W^{1,1}(\frac{3}{4}B)$, there exists a constant $C>0$ such that
$$\begin{array}{ll}
& \left( \displaystyle \dashint_{\frac{3}{4}B} G^-\left(\min\{\frac{(v^+ - l)^+}{R},\frac{k}{R}\}\right)^\chi \, \mathrm{d}x\right)^\frac{1}{\chi}\\ 
& \leq C R \displaystyle \dashint_{\frac{3}{4}B} \left|\nabla G^-\left(\min\{\frac{(v^+ - l)^+}{R},\frac{k}{R}\}\right)\right| \, \mathrm{d}x\\
& \leq C R\displaystyle \dashint_{\frac{3}{4}B} \tilde{g}\left(\min\{\frac{(v^+ - l)^+}{R},\frac{k}{R}\}\right) \left|\nabla \min\{\frac{(v^+ - l)^+}{R},\frac{k}{R}\} \right|  \, \mathrm{d}x \\
& \leq \displaystyle C\displaystyle \dashint_{\frac{3}{4}B} \tilde{g}\left(\min\{\frac{(v^+ - l)^+}{R},\frac{k}{R}\}\right) \left|\nabla \min\{(v^+ - l)^+,k\} \right|  \, \mathrm{d}x.
\end{array}$$
where $\chi := 1^* = \displaystyle \frac{n}{n-1}$.\\ 
Using inequality $(2.1)$, for $a=\min\{\displaystyle \frac{(v^+ - l)^+}{R},\displaystyle \frac{k}{R}\}$ and $b=2g^0C|\nabla \min\{(v^+ - l)^+,k\}|$, we have
$$\begin{array}{ll}
& \left( \displaystyle \dashint_{\frac{3}{4}B} G^-\left(\min\{\frac{(v^+ - l)^+}{R},\frac{k}{R}\}\right)^\chi \, \mathrm{d}x\right)^\frac{1}{\chi}\\
& \leq \displaystyle \frac{1}{2g^0}\dashint_{\frac{3}{4}B} \tilde{g}\left(\min\{\frac{(v^+ - l)^+}{R},\frac{k}{R}\}\right) \left|\min\{\frac{(v^+ - l)^+}{R},\frac{k}{R}\} \right|  \, \mathrm{d}x\\
&\; \;  + \displaystyle \frac{1}{2g^0} \dashint_{\frac{3}{4}B} \tilde{g}\left( C\left| \nabla \min\{(v^+ - l)^+,k\}\right|\right) C\left| \nabla \min\{(v^+ - l)^+,k\}\right|  \, \mathrm{d}x\\
& \leq \displaystyle \frac{1}{2}\dashint_{\frac{3}{4}B} G^-\left(\min\{\frac{(v^+ - l)^+}{R},\frac{k}{R}\}\right) \, \mathrm{d}x + \displaystyle \frac{C^{g^0}}{2} \dashint_{\frac{3}{4}B} G^-\left( \left| \nabla \min\{(v^+ - l)^+,k\}\right|\right)  \, \mathrm{d}x.\\
\end{array}$$
As,
$$\displaystyle \dashint_{\frac{3}{4}B} G^-\left(\min\{\frac{(v^+ - l)^+}{R},\frac{k}{R}\}\right) \, \mathrm{d}x \leq \left( \displaystyle \dashint_{\frac{3}{4}B} G^-\left(\min\{\frac{(v^+ - l)^+}{R},\frac{k}{R}\}\right)^\chi \, \mathrm{d}x\right)^\frac{1}{\chi}.$$
Then, there exists a constant $C>0$ such that
\begin{equation}
\begin{array}{ll}
& \left( \displaystyle \dashint_{\frac{3}{4}B} G^-\left(\min\{\frac{(v^+ - l)^+}{R},\frac{k}{R}\}\right)^\chi \, \mathrm{d}x\right)^\frac{1}{\chi}\\
& \leq C \displaystyle \dashint_{\frac{3}{4}B} G^-\left( \left| \nabla \min\{(v^+ - l)^+,k\}\right|\right)  \, \mathrm{d}x.
\end{array}
\end{equation}
Combining the inequalities $(5.8)$ and $(5.9)$, we obtain
\begin{equation}
 \left(\displaystyle \dashint_{\frac{3}{4}B} G^-\left(\min\{\frac{(v^+ - l)^+}{R},\frac{k}{R}\}\right)^\chi \, \mathrm{d}x\right)^\frac{1}{\chi} \leq  C k \displaystyle \frac{\mu(\overline{B})}{R^{n}}.
\end{equation}
Otherwise, by Lemma $5.3$, we have
$$\begin{array}{ll}
& \left( \displaystyle \dashint_{\frac{3}{4}B} G^-\left(\min\{\frac{(v^+ - l)^+}{R},\frac{k}{R}\}\right)^\chi \, \mathrm{d}x\right)^\frac{1}{\chi}\\ 
& \geq \left(\displaystyle \dashint_{ \{x \in \frac{1}{2}B \, : \,  (v^+ - l)^+ \geq k \} } G^-\left(\min\{\frac{(v^+ - l)^+}{R},\frac{k}{R}\}\right)^\chi \, \mathrm{d}x\right)^\frac{1}{\chi}\\
& \geq G^-\left(\displaystyle \frac{k}{R}\right) \displaystyle \frac{\left|\{ x \in \frac{1}{2}B \, : \,  (v^+ - l)^+ \geq k\}\right|^\frac{1}{\chi}}{|\frac{1}{2}B|^\frac{1}{\chi}}\\
& \geq \displaystyle \frac{1}{g^0}\frac{k}{R} \tilde{g}\left(\displaystyle \frac{k}{R}\right) \displaystyle \frac{\left|\{ x \in \frac{1}{2}B \, : \,  (v^+ - l)^+ \geq k \}\right|^\frac{1}{\chi}}{|\frac{1}{2}B|^\frac{1}{\chi}}.
\end{array}$$
Then, by the inequality $(5.10)$, we have
$$\tilde{g}\left(\frac{k}{R}\right) \frac{\left|\{ x \in \frac{1}{2}B \, : \,  (v^+ - l)^+ \geq k \}\right|^\frac{1}{\chi}}{|\frac{1}{2}B|^\frac{1}{\chi}} \leq C \frac{\mu(\overline{B})}{R^{n-1}}.$$
So,
$$\tilde{g}\left(\frac{k}{R}\right) \left(\frac{\left|\{ x \in \frac{1}{2}B \, : \,  (v^+ - l)^+ \geq k \}\right|^\frac{1}{\chi^\prime}}{|\frac{1}{2}B|^\frac{1}{\chi^\prime}}\right)^{g_0-1} \leq C \frac{\mu(\overline{B})}{R^{n-1}}.$$
where $\chi^\prime = \chi(g_0-1)$. Using inequality $(2.3)$, we get
$$\tilde{g}\left(\frac{k}{R} \frac{\left|\{ x \in \frac{1}{2}B \, : \,  (v^+ - l)^+ \geq k \}\right|^\frac{1}{\chi^\prime}}{|\frac{1}{2}B|^\frac{1}{\chi^\prime}}\right) \leq C \frac{\mu(\overline{B})}{R^{n-1}}.$$
Using Lemma $5.9$ and definition of Lorentz norms, we obtain
\begin{equation}
\tilde{g}\left(\frac{\norme{(v^+ - l)^+}^{\chi^\prime,\infty}_{\frac{1}{2}B}}{R|\frac{1}{2}B|^\frac{1}{\chi^\prime}}\right) \leq  C \frac{\mu(\overline{B})}{R^{n-1}}.
\end{equation}
Since,
$$1 \leq \displaystyle \frac{\norme{(v^+ - l)^+}^{\chi^\prime,\infty}_{\frac{1}{2}B}}{R|\frac{1}{2}B|^\frac{1}{\chi^\prime}} + 1 \leq \displaystyle \frac{\norme{v^+}^{\chi^\prime,\infty}_{\frac{1}{2}B}}{R|\frac{1}{2}B|^\frac{1}{\chi^\prime}} + 1 \leq \displaystyle \frac{\frac{\norme{u^+}^{\chi^\prime,\infty}_{\frac{1}{2}B}}{|\frac{1}{2}B|^\frac{1}{\chi^\prime}} + 1}{R}.$$
Then, by Lemma $5.4$, there exists a constant C depend of $\displaystyle \frac{\norme{u^+}^{\chi^\prime,\infty}_{\frac{1}{2}B}}{|\frac{1}{2}B|^\frac{1}{\chi^\prime}}$ such that
$$G\left( x_0,\frac{\norme{(v^+ - l)^+}^{\chi^\prime,\infty}_{\frac{1}{2}B}}{R|\frac{1}{2}B|^\frac{1}{\chi^\prime}} + 1\right) \leq C G^-\left(\frac{\norme{(v^+ - l)^+}^{\chi^\prime,\infty}_{\frac{1}{2}B}}{R|\frac{1}{2}B|^\frac{1}{\chi^\prime}}+1\right).$$
So, by the condition $(SC)$ and Lemma $5.3$, we have
$$g\left( x_0,\frac{\norme{(v^+ - l)^+}^{\chi^\prime,\infty}_{\frac{1}{2}B}}{R|\frac{1}{2}B|^\frac{1}{\chi^\prime}}+1 \right) \leq C \tilde{g}\left(\frac{\norme{(v^+ - l)^+}^{\chi^\prime,\infty}_{\frac{1}{2}B}}{R|\frac{1}{2}B|^\frac{1}{\chi^\prime}}+1\right).$$
Hence, by the condition $(A_0)$ and the inequality $(5.11)$, we have
$$g\left( x_0,\frac{\norme{(v^+ - l)^+}^{\chi^\prime,\infty}_{\frac{1}{2}B}}{R|\frac{1}{2}B|^\frac{1}{\chi^\prime}} \right) \leq C \tilde{g}\left(\frac{\norme{(v^+ - l)^+}^{\chi^\prime,\infty}_{\frac{1}{2}B}}{R|\frac{1}{2}B|^\frac{1}{\chi^\prime}}\right) + C\tilde{g}(1) \leq C\left(\frac{\mu(\overline{B})}{R^{n-1}} + 1\right).$$
So,
$$\frac{\norme{(v^+ - l)^+}^{\chi^\prime,\infty}_{\frac{1}{2}B}}{|\frac{1}{2}B|^\frac{1}{\chi^\prime}} \leq  R g^{-1}\left( x_0,C\left(\frac{\mu(\overline{B})}{R^{n-1}}+1\right)\right).$$
From Remark $2.2$, the condition$(SC)$, inequalities $(2.2)$, $(2.3)$ and the condition $(A_0)$, we have
$$\frac{\norme{(v^+ - l)^+}^{\chi^\prime,\infty}_{\frac{1}{2}B}}{|\frac{1}{2}B|^\frac{1}{\chi^\prime}} \leq CR g^{-1}\left( x_0,\left(\frac{\mu(\overline{B})}{R^{n-1}}\right) \right) + CRg^{-1}(1) \leq  CR g^{-1}\left( x_0,\left(\frac{\mu(\overline{B})}{R^{n-1}}\right)\right) + CR.$$
Using Proposition $5.8$ and inequality $(5.7)$, we obtain
$$\begin{array}{ll}
 \displaystyle \frac{1}{|\frac{1}{2}B|^{\frac{1}{\chi^\prime}}}\norme{v^+}^{\chi^\prime,\infty}_{\frac{1}{2}B} & \leq l + \displaystyle \frac{1}{|\frac{1}{2}B|^{\frac{1}{\chi^\prime}}}\norme{(v^+ - l)^+}^{\chi^\prime,\infty}_{\frac{1}{2}B}\\
 & \leq C\left(\displaystyle \dashint_{A} \overline{v}^\gamma \, \mathrm{d}x\right)^\frac{1}{\gamma} + C R g^{-1}\left( x_0,\displaystyle \frac{\mu(\overline{B})}{R^{n-1}}\right) + CR.
\end{array}$$
Since $u=v$ in $\displaystyle \frac{1}{2}B$ and $v \leq u$ in $B$, we get the inequality $(5.6)$
$$ \displaystyle \frac{1}{|\frac{1}{2}B|^{\frac{1}{\chi^\prime}}}\norme{u}^{\chi^\prime,\infty}_{\frac{1}{2}B} \leq C\left(\displaystyle \dashint_{B \setminus \frac{1}{2}B} u^\gamma \, \mathrm{d}x\right)^\frac{1}{\gamma} + C R g^{-1}\left( x_0,\frac{\mu(\overline{B})}{R^{n-1}}\right) + CR.$$
\textbf{Iteration step:} We show the following inequality
\begin{equation}
u(x_0) \leq C\left( R + \left(\dashint_{B\setminus\frac{1}{2}B} u^\gamma \, \mathrm{d}x\right)^\frac{1}{\gamma} +  W^{\mu}_{G(\cdot)}(x_0,2R)\right).
\end{equation}
Let $B_0=B_R$, for $j = 0,1,...,$ take  $R_j = 2^{-j}R, \; B_j = B_{R_j}$. Also, for $\delta \in (0,1)$  we consider a sequence
$$l_0 := 0, \; l_{j+1} := l_j + \frac{1}{\delta^\frac{1}{\chi^\prime}}\frac{1}{|B_{j+1}|^\frac{1}{\chi^\prime}}\norme{u-l_j}^{\chi^\prime,\infty}_{B_{j+1}}.$$
By the definition of $l_j$ and the inequality $(5.5)$ for $(u-l_j)$, we have
\begin{equation}
\begin{array}{ll}
& l_{j+1} - l_j \\
& = \displaystyle \frac{1}{\delta^\frac{1}{\chi^\prime}}\frac{1}{|B_{j+1}|^\frac{1}{\chi^\prime}}\norme{u-l_j}^{\chi^\prime,\infty}_{B_{j+1}}\\
& \leq  \displaystyle \frac{C}{\delta^\frac{1}{\chi^\prime}}\left( \left( \displaystyle \dashint_{B_j\setminus B_{j+1}} (u - l_j)^{\gamma} \, \mathrm{d}x\right)^\frac{1}{\gamma} + R_j g^{-1}\left( x_0,\displaystyle \frac{\mu(\overline{B_j})}{R_j^{n-1}}\right) + R_j \right)\\
& \leq  \displaystyle \frac{C}{\delta^\frac{1}{\chi^\prime}}\left( \left( \displaystyle \dashint_{B_j} (u  -l_j)^{\gamma} \, \mathrm{d}x\right)^\frac{1}{\gamma} + R_j g^{-1}\left( x_0,\displaystyle \frac{\mu(\overline{B_j})}{R_j^{n-1}}\right) + R_j\right).
\end{array}
\end{equation}
If, we choose $\gamma \leq \chi^\prime $, then by Proposition $5.8$, for $q=\displaystyle \frac{\chi^\prime}{\gamma}$, we have
$$\begin{array}{ll}
\left(\displaystyle \dashint_{B_j} (u - l_j)^{\gamma} \, \mathrm{d}x\right)^{\frac{1}{\gamma}} & = \left( \displaystyle \dashint_{\{x \in B_j\; ,: \; u > l_j \}} (u - l_j)^{\gamma} \, \mathrm{d}x\right)^{\frac{1}{\gamma}}\\
& \leq C  \frac{1}{|B_j|^\frac{1}{\gamma}} \left|\{ x\in B_j \; : \; u(x) \geq l_j \}\right|^{\frac{1}{(\frac{\chi^\prime}{\gamma})^*}\frac{1}{\gamma}}\left(\norme{(u - l_j)^{\gamma}}^{\frac{\chi^\prime}{\gamma},\infty}_{B_{j}}\right)^{\frac{1}{\gamma}}\\
& \leq C  \frac{1}{|B_j|^\frac{1}{\gamma}} \left|\{ x \in B_j \; : \; u(x) \geq l_j \}\right|^{\frac{1}{(\frac{\chi^\prime}{\gamma})^*}\frac{1}{\gamma}}\norme{u - l_j}^{\chi^\prime,\infty}_{B_{j}}.
\end{array}$$
Since,
$$\begin{array}{ll}
|\{x\in B_j \; : \; u(x) \geq l_j \}| & = |\{x\in B_j \; : \; u(x)- l_{j-1} \geq l_j - l_{j-1}\}| \\
&= \left(|\{x\in B_j \; : \; \frac{u(x)- l_{j-1}}{l_j - l_{j-1}} \geq 1 \}|^{\frac{1}{\chi^\prime}}\right)^{\chi^\prime}\\
&\leq \left(\displaystyle\frac{1}{l_j - l_{j-1}}\norme{u-l_{j-1}}^{\chi^\prime,\infty}_{B_{j}}\right)^{\chi^\prime}\\
& \leq \delta |B_j|.
\end{array}$$
Moreover, by the definition of $l_j$, we have
$$\norme{u-l_{j-1}}^{\chi^\prime,\infty}_{B_j}= (l_j - l_{j-1})\delta^\frac{1}{\chi^\prime} |B_{j}|^\frac{1}{\chi^\prime}.$$
Then,
$$\begin{array}{ll}
\left(\displaystyle \dashint_{B_j} (u - l_j)^{\gamma} \, \mathrm{d}x\right)^{\frac{1}{\gamma}} & \leq C\delta^{\frac{1}{\gamma}\frac{1}{(\frac{\chi^\prime}{\gamma})^*}}\frac{1}{|B_j|^\frac{1}{\chi^\prime}}\norme{u - l_j}^{\chi^\prime,\infty}_{B_{j}}\\
& \leq C\delta^{\frac{1}{\gamma}\frac{1}{(\frac{\chi^\prime}{\gamma})^*}}\frac{1}{|B_j|^\frac{1}{\chi^\prime}}\norme{u - l_{j-1}}^{\chi^\prime,\infty}_{B_{j}}\\
& \leq C\delta^{\frac{1}{\gamma}\frac{1}{(\frac{\chi^\prime}{\gamma})^*}+\frac{1}{\chi^\prime}}(l_j - l_{j-1})\\
& \leq C \delta^\frac{1}{\gamma}(l_j - l_{j-1}).
\end{array}$$
Hence, by the inequality $(5.13)$, we have
$$l_{j+1} - l_j \leq C\delta^{\frac{1}{\gamma}-\frac{1}{\chi^\prime}}(l_j - l_{j-1}) + \frac{C}{\delta^\frac{1}{\chi^\prime}} R_j g^{-1}\left( x_0,\frac{\mu(\overline{B_j})}{R_j^{n-1}}\right) + CR_j.$$
We choose $\delta > 0$ such that $C\delta^{\frac{1}{\gamma} - \frac{1}{\chi^\prime}} \leq \displaystyle \frac{1}{2}$, we get
$$l_{j+1} - l_j \leq \frac{1}{2}(l_j - l_{j-1}) + C R_j g^{-1}\left( x_0,\frac{\mu(\overline{B_j})}{R_j^{n-1}}\right) + CR_j.$$
Hence,
$$l_{k+1} - l_1 = \displaystyle \sum_{j=1}^{j=k}(l_{j+1} - l_{j}) \leq \frac{1}{2}(l_k - l_0) +  C\sum_{j=1}^{j=k} R_jg^{-1}\left( x_0,\frac{\mu(\overline{B_j})}{R_j^{n-1}}\right) +  C\sum_{j=1}^{j=k} R_j.$$
By the definition of $l_1$ and inequality $(5.6)$, we have
$$\begin{array}{ll}
l_1 & = \displaystyle \frac{1}{\delta^\frac{1}{\chi^\prime}}\frac{1}{|B_{1}|^\frac{1}{\chi^\prime}}\norme{u}^{\chi^\prime,\infty}_{B_{1}}\\
& \leq \displaystyle \frac{C}{\delta^\frac{1}{\chi^\prime}}\left(\left(\displaystyle \dashint_{B_0\setminus B_{1}} u^{\gamma} \, \mathrm{d}x\right)^\frac{1}{\gamma} + R_0 g^{-1}\left( x_0,\frac{\mu(\overline{B_0})}{R_0^{n-1}}\right) + R_0 \right).
\end{array}$$
So,
$$l_{k+1} \leq \frac{1}{2}l_k + C\left(\displaystyle \dashint_{B_0\setminus B_{1}} u^{\gamma} \, \mathrm{d}x\right)^\frac{1}{\gamma}  + C\sum_{j=0}^{j=k} R_jg^{-1}\left( x_0,\frac{\mu(\overline{B_j})}{R_j^{n-1}}\right) +  C\sum_{j=0}^{j=k} R_j.$$
Taking the limit $k \rightarrow \infty$, then
$$\frac{1}{2}l_\infty \leq C\left(\displaystyle \dashint_{B_0\setminus B_{1}} u^{\gamma} \, \mathrm{d}x\right)^\frac{1}{\gamma}  + C\sum_{j=0}^{\infty} R_jg^{-1}\left( x_0,\frac{\mu(\overline{B_j})}{R_j^{n-1}}\right) +  C \displaystyle \sum_{j=0}^{\infty} R_j.$$
where $l_\infty = \lim_{k\rightarrow \infty}l_k$. Or, by Proposition $5.8$ and the definition of $l_j$, we have
$$\begin{array}{ll}
 \frac{1}{|B_{j}|^\frac{1}{\chi^\prime}}\norme{u}^{\chi^\prime,\infty}_{B_{j}} & \leq l_\infty + \frac{1}{|B_{j}|^\frac{1}{\chi^\prime}}\norme{u-l_\infty}^{\chi^\prime,\infty}_{B_{j}}\\
& \leq l_\infty + \frac{1}{|B_{j}|^\frac{1}{\chi^\prime}}\norme{u-l_{j-1}}^{\chi^\prime,\infty}_{B_{j}} \\
& \leq l_\infty + \delta^\frac{1}{\chi^\prime}(l_j - l_{j-1}).
\end{array}$$
Taking the upper limit, we obtain
$$\begin{array}{ll}
\limsup_{j \rightarrow \infty} \frac{1}{|B_{j}|^\frac{1}{\chi^\prime}}\norme{u}^{\chi^\prime,\infty}_{B_{j}} & \leq l_\infty \\
& \leq C\left( R + \left(\displaystyle \dashint_{B_0\setminus B_{1}} u^{\gamma} \, \mathrm{d}x\right)^\frac{1}{\gamma} +  \displaystyle \sum_{j=0}^{\infty} R_jg^{-1}\left( x_0,\frac{\mu(\overline{B_j})}{R_j^{n-1}}\right)\right).
\end{array}$$
As $u$ is lower semicontinuous at $x_0$, then
$$u(x_0) \leq C \left(R + \left( \displaystyle \dashint_{B_0\setminus B_{1}} u^{\gamma} \, \mathrm{d}x\right)^\frac{1}{\gamma} + \sum_{j=0}^{\infty} R_jg^{-1}\left( x_0,\frac{\mu(\overline{B_j})}{R_j^{n-1}}\right)\right).$$
Or, by Remark $2.2$, the condition$(SC)$ and inequalities $(2.2)$, $(2.3)$, we have
$$R_jg^{-1}\left( x_0,\frac{\mu(\overline{B_j})}{R_j^{n-1}}\right) = g^{-1}\left( x_0,\frac{\mu(\overline{B_j})}{R_j^{n-1}}\right)\int_{R_j}^{2R_j}  \, \mathrm{d}s \leq C \int_{R_j}^{2R_j}  \, g^{-1} \left(x_0,\frac{\mu(B(x_0,s)}{s^{n-1}}\right) \mathrm{d}s.$$
So,
$$\sum_{j=0}^{\infty} R_jg^{-1}\left( x_0,\frac{\mu(\overline{B_j})}{R_j^{n-1}}\right) \leq C \int_{0}^{2R}  \, g^{-1}\left( x_0,\frac{\mu(B(x_0,s)}{s^{n-1}}\right) \mathrm{d}s.$$
Therefore,
$$ u(x_0) \leq C\left( R + \left(\displaystyle \dashint_{B_0\setminus B_1} u^{\gamma} \, \mathrm{d}x\right)^\frac{1}{\gamma} + W^{\mu}_{G(\cdot)}(x_0,2R) \right).$$
This complete the proof. 
\end{proof}

\begin{rms}
Let $G(\cdot) \in \Phi(\Omega)$ satisfies $(SC)$ and $(A_0)$ 
\begin{enumerate}
\item[1)] If $g_0 \leq n$, by Corollary $3.7.9$ in $\cite{ref9}$ and Proposition $5.8$, we have $L^{G(\cdot)}(B) \subset L^{g_0}(B) \subset L^{\chi^\prime,\infty}(B)$.
\item[2)] If $u$ be a nonnegative supersolution to $(3.1)$ in $B_{2R}$ satisfies the assumptions of Lemma $3.4$, then $u \in L^{\chi^\prime,\infty}(B_R)$. Indeed,  by the weak Harnack inequality, we have $u \in L^\gamma(B_{R})$ for any $\gamma < \gamma_0$ with
$$\gamma_0 := 
\begin{cases}
\frac{n(g_0-1)}{n-g_0} & \text{if} \; g_0 < n \\
\infty & \text{if} \; g_0 \geq n.
\end{cases}$$
As $\chi^\prime < \gamma_0$, then $u \in L^{\chi^\prime,\infty}(B_R)$.
\item[3)] Let us mention, the result of Theorem $5.10$ still holds for a signed Radon measure $\mu$ if $|\mu| \in W^{1,G(\cdot)}_0(\Omega)$ and the solutions of equation $(3.2)$ are locally bounded.
\end{enumerate}
\end{rms}

Finally, combining Theorem $5.6$, Theorem $5.10$ and weak Harnack inequality (see Lemma $3.4$), we obtain the following main result.

\begin{Th}
Let $G(\cdot) \in \Phi(\Omega)$ satisfies $(SC)$ and $(A_0)$. Let $u$ be a nonnegative supersolution to $(3.1)$ in $B_{2R} = B(x_0, 2R)$ and  $\mu$ is the associated Radon measure of $u$. If $u$ is lower semicontinuous at $x_0$ and that one of the following holds
\begin{enumerate}
\item[(1)] $G(\cdot)$ satisfies $(A_{1,s_*})$ and $\norme{u}_{L^s(B_{2R})} \leq d$, where $s_*= \frac{ns}{n+s}$ and  $s > \max\{\frac{n}{g_0},1\}(g^0-g_0)$.
\item[(2)] $G(\cdot)$ satisfies $(A_{1})$, $\norme{u}_{W^{1,G(\cdot)}(B_{2R})} \leq d$ and $\frac{ng_0}{n-g_0} > g^0$.
\end{enumerate} 
Then there exists a positive constant $C$ such that
\begin{equation}
\frac{1}{C} W^{\mu}_{G(\cdot)}(x_0,R) - 2R \leq u(x_0) \leq C\left(R + \inf_{B}u +  W^{\mu}_{G(\cdot)}(x_0,2R) \right).
\end{equation}
\end{Th}

\begin{rma}
Noting that the result of the theorem $5.12$ can be extend, in similar way as in $\cite{ref8}$, to superharmonic functions related to equation $(3.1)$.
\end{rma}

As in $\cite{ref27}$, we apply our Wolff potential estimates to prove that the supersolutions satisfy a Harnack inequality and local Hölder continuity under an assumption on the growth order of the measure $\mu$.

\begin{Th}
Let $G(\cdot)$, $u$ and $\mu$ satisfy the assumptions of Theorem $5.12$. If there exist $\alpha \in (0,1)$ and $\tilde{C}>0$ such that
\begin{equation}
\mu(B(x,r) \leq \tilde{C}r^{n -1}g(x,r^{\alpha-1}),
\end{equation}
whenever $x \in B(x_0,R)$ and $0 < r < 4R$. Then there exists $C > 0$ such that
$$\sup_{B(x_0,R)}u(x) \leq C(\inf_{B(x_0,R)}u + R^\alpha).$$
\end{Th}

\begin{proof}
If $x \in B(x_0,R)$, then, by the upper estimate from Theorem $5.12$, we have
$$\begin{array}{ll}
u(x) \leq & C(\inf_{B(x,2R)}u + R +  W^{\mu}_{G(\cdot)}(x,2R))\\
& \leq C\left(\inf_{B(x_0,R)}u + R + \displaystyle \int_{0}^{2R}  \, g^{-1}\left(x,\frac{s^{n-1}g(x,s^{\alpha - 1})}{s^{n-1}}\right) \mathrm{d}s\right)\\
& \leq C\left(\inf_{B(x_0,R)}u + R^\alpha\right).
\end{array}$$
Hence, we take supremum over in $B(x_0,R)$, we obtain the Harnack inequality. 
\end{proof}

\begin{cor}
Let $G(\cdot)$, $u$ and $\mu$ satisfy the assumptions of Theorem $5.12$. Then $\mu$ satisfies $(5.14)$ if and only if $u$ is locally Hölder continuous.
\end{cor}

\begin{proof}
Let $u$ be a nonnegative supersolution to $(3.1)$ in $B_{2R} = B(x_0, 2R)$ and  $\mu$ is the associated Radon measure of $u$. If the condition $(5.14)$ holds, then by Theorem $5.14$, $u$ satisfies the Harnack inequality. So, by standard arguments as in $\cite{ref13}$, we get locally Hölder continuity of $u$. \\
For the other implication, let $u \in C^{0,\alpha}_{loc}(\Omega)$, we apply the lower estimate from Theorem $5.12$ for $u - \inf_{B(x,2r)}u$, we obtain
$$\begin{array}{ll}
\displaystyle g^{-1}\left( x,\displaystyle \frac{\mu(B(x,r)}{r^{n-1}}\right) & \leq  \displaystyle \frac{C}{r}\int_{r}^{2r}  \, g^{-1} \left(x,\frac{\mu(B(x,s)}{s^{n-1}}\right) \mathrm{d}s\\
& \leq \displaystyle \frac{C}{r}\left( u(x) - \inf_{B(x,2r)}u + r \right) \\
& \leq Cr^{\alpha - 1}.
\end{array}$$
Then
$$\mu(B(x,r) \leq Cr^{n-1}g(x,r^{\alpha - 1}).$$
Therefore, we have the equivalent. 
\end{proof}


\begin{thebibliography}{99}

\bibitem{ref1} Alkhutov Yu., Krasheninnikova O.: Continuity at boundary points of solutions of quasilinear elliptic equations with nonstandard growth. Izvestiya RAN, seriya matematicheskaya. 68, 3-60 (2004) 
\bibitem{ref2} Benyaiche A., Harjulehto P., Peter H\"ast\"o P., Karppinen  A.: The weak Harnack inequality for unbounded supersolutions of equations with generalized Orlicz growth. Journal of Differential Equations. To appear
\bibitem{ref3} Benyaiche A., Khlifi I.: Harnack Inequality for Quasilinear Elliptic Equations in Generalized Orlicz-Sobolev Spaces. Potential Analysis. 53, 631–643 (2020)
\bibitem{ref4} Benyaiche A., Khlifi I.: Sobolev–Dirichlet problem for quasilinear elliptic equations in generalized Orlicz–Sobolev spaces. Positivity (2020). https://doi.org/10.1007/s11117-020-00789-z
\bibitem{ref5} Björn A., Björn J.: Nonlinear Potential Theory on Metric Spaces. European Mathematical Society, Zürich. (2011).
\bibitem{ref6} Chlebicka I., Giannetti F., Zatorska-Goldstein A.: Wolff potentials and local behaviour of solutions to measure data elliptic problems with Orlicz growth. (2020), arXiv:2006.02172.
\bibitem{ref7} Grafakos L.: Classical Fourier Analysis. Graduate Texts in Mathematics. Springer, New York. (2008).
\bibitem{ref8} Hara T.: Wolff potential estimates for Cheeger p-harmonic functions. Collectanea Mathematica. 69, 407–426 (2018)
\bibitem{ref9} Harjulehto P., H\" ast\"o P.: Orlicz Spaces and Generalized Orlicz Spaces. Springer-Verlag, Cham. (2019)
\bibitem{ref10} Harjulehto P., H\"ast\"o P.: Boundary regularity under generalized growth conditions. Zeitschrift für Analysis und ihre Anwendungen. 38, 73-96 (2019) 
\bibitem{ref11} Harjulehto P., H\" ast\"o P., Toivanen O.: H\"older regularity of quasiminimizers under generalized growth conditions. Calculus of Variations and Partial Differential Equations. 56, article 22 (2017)
\bibitem{ref12} Hedberg L.I., Wolff Th.H.: Thin Sets in Nonlinear Potential Theory. Annales de l'Institut Fourier. 33, 161-187 (1983)
\bibitem{ref13} Heinonen J., Kilpel\" ainen T., Martio O.: Nonlinear Potential Theory of Degenerate Elliptic Equations. Clarendon Press. (1993)
\bibitem{ref27} Kilpeläinen T.: Hölder continuity of solutions to quasilinear elliptic equations involving measures. Potential Anal 3, 265–272 (1994). https://doi.org/10.1007/BF01468246
\bibitem{ref14} Kilpeläinen T., Malý J.: Degenerate elliptic equations with measure data and nonlinear potentials. Annali della Scuola normale superiore di Pisa, Classe di scienze. 19, 591-613 (1992)
\bibitem{ref15} Kilpeläinen T., Malý J.: The Wiener test and potential estimates for quasilinear elliptic equations. Acta Mathematica. 172, 137-161 (1994)
\bibitem{ref16} Korte R., Kuusi T.: A note on the Wolff potential estimate for solutions to elliptic equations involving measures. Advances in Calculus of Variations. 3, 99-113 (2010)
\bibitem{ref29} Kuusi, T., Mingione, G.: Guide to nonlinear potential estimates.Bulletin of Mathematical Sciences. 4, 1–82 (2014). https://doi.org/10.1007/s13373-013-0048-9
\bibitem{ref17} Lieberman G.M.: The natural generalization of the natural conditions of Ladyzhenskaya and Ural' tseva for elliptic equations. Communications in Partial Differential Equations. 16, 311–361 (1991)
\bibitem{ref18} Lindqvist P., Martio O.: Two theorems of N. Wiener for solutions of quasilinear elliptic equations. Acta Mathematica. 155, 153-171 (1985)
\bibitem{ref19} Lions J.L.: Quelques Méthodes De Résolution Des Problèmes Aux Limites Non Linéaires. Dunod, GauthierVillars, Paris. (1969)
\bibitem{ref20} Lukkari T., Maeda F.Y., Marola N.: Wolff potential estimates for elliptic equations with nonstandard growth and applications. Forum Mathematicum. 22, 1061-1087 (2010)
\bibitem{ref21} Malý J.: Wolff potential estimates of superminimizers of Orlicz type Dirichlet integrals. Manuscripta Mathematica. 110, 513-525 (2003)
\bibitem{ref22} Maz'ya V. G.: On the continuity at a boundary point of solutions of quasilinear elliptic equations. Vestnik Leningrad Univ. 25, 42-55 (1970) (Russian)
\bibitem{ref23} Mikkonen P.: On the Wolff potential and quasilinear elliptic equations involving measures. Annales Academiae Scientiarum Fennicae, Mathematica, Dissertationes. 104, 1–71 (1996)
\bibitem{ref24} Mingione G., Palatucci G.: Developments and perspectives in Nonlinear Potential Theory. Nonlinear Analysis. 194, article 111452 (2020)
\bibitem{ref25} Musielak J.: Orlicz spaces and modular spaces. Springer, Berlin. (1983)
\bibitem{ref26} Trudinger N., Wang X.J.: On the weak continuity of elliptic operators and applications to potential theory. American Journal of Mathematics. 124, 369-410 (2002)
\bibitem{ref28} Widman K.O.: Hölder continuity of solutions of elliptic equations. Manuscripta Mathematica. 5, 299-308  (1971)

\end{thebibliography}
\end{document}